\documentclass[12pt,leqno]{amsart}
\usepackage{amsmath,amssymb,amsfonts}
\usepackage{eucal,enumerate}
\usepackage{graphicx}
\usepackage{color}
\usepackage{hyperref}

\newtheorem{theorem}{Theorem}
\newtheorem{proposition}[theorem]{Proposition}
\newtheorem{corollary}[theorem]{Corollary}
\newtheorem{lemma}[theorem]{Lemma}

\renewcommand{\phi}{\varphi}                 
\renewcommand{\epsilon}{\varepsilon}

\renewcommand\Theta{\varTheta}
\newcommand\setm{\smallsetminus}
\renewcommand\bar\widehat

\begin{document}

\title{The Cycle Counts of Graphs}
\author{Ryan McCulloch}
\address{Department of Mathematics and Statistics, Binghamton University, Binghamton, NY 13902-6000, U.S.A.}
\email{rmccullo1985@gmail.com}
\author{Brendan D.\ McKay}
\address{School of Computing, Australian National University, Canberra, ACT 2601, Australia}
\email{brendan.mckay@anu.edu.au}
\author{Alireza Salahshoori}
\address{Department of Mathematics and Statistics, Binghamton University, Binghamton, NY 13902-6000, U.S.A.}
\email{asalahs1@binghamton.edu}
\author{Thomas Zaslavsky}
\address{Department of Mathematics and Statistics, Binghamton University, Binghamton, NY 13902-6000, U.S.A.}
\email{zaslav@math.binghamton.edu}

\date{November 26, 2025}

\begin{abstract}
We prove that an inseparable graph can have any positive number of cycles with the six exceptions 2, 4, 5, 8, 9, 16, and that an inseparable cubic graph has the additional exceptions 1 and 13.  The exceptions for simple inseparable cubic graphs are unknown.
\end{abstract}


\maketitle

A \emph{cycle count number} is a positive integer that is the number of cycles in some inseparable graph.
The principal result of this paper is as follows.

\begin{theorem}\label{conjecture}
The only positive integers that are not cycle count numbers are $2$, $4$, $5$, $8$, $9$, and $16$.
\end{theorem}

Since the number of cycles is unchanged upon subdividing edges, or the reverse operation, Theorem~\ref{conjecture} is equally true for multigraphs and simple graphs.
In this article we write ``graph'' to allow multiple edges but not loops, and ``simple graph'' when neither  loops nor multiple edges are allowed.  We will only consider \emph{inseparable} graphs, which are those without cutpoints.  For loop-free graphs with more than two vertices, inseparability is the same as 2-connectivity.  The following is an
example of a theorem that is true for graphs but not for simple graphs (see Section~\ref{s:conjectures}). 
    
\begin{theorem}\label{cubic}
The only positive integers that are not cycle counts of inseparable cubic graphs are $1$, $2$, $4$, $5$, $8$, $9$, $13$, and $16$.
The same positive integers are the only ones that are not cycle counts of inseparable planar cubic graphs.
\end{theorem}

The numbers in Theorems~\ref{conjecture} and~\ref{cubic} are sequences A385523 and A385524 in the Online Encyclopedia of Integer Sequences \cite{oeis}.

The proof is in three parts.  The first part is a detailed analysis of graphs with small cycle counts, proving nonexistence of the exceptional cases.  The second part establishes a connection to subtree counts of  trees, enabling us to prove all cases except a finite list ending at 89.  The third part presents the results of a computer search that found graphs for all numbers in that finite list.

Let $c(G)$ be the number of cycles in $G$.  This number depends only on the homeomorphism type of $G$, so we may freely introduce new vertices by subdividing edges, or execute the reverse operation of suppressing divalent vertices.  
We denote by $\bar{G}$ any graph that is homeomorphic to $G$ (i.e., they have isomorphic subdivisions). 

An \emph{ear} in a graph is a path (with distinct endpoints) of positive length whose internal vertices, if any, have degree 2, and whose endpoints have degree greater than 2.
When adding an ear, we allow ourselves to subdivide edges to become the endpoints of the ear.

Our approach will use the following famous theorem of Whitney.
\begin{theorem}[{\cite[Theorem 19]{Whitney}}]\label{whitney}
Adding an ear to an inseparable graph gives an inseparable graph.
Moreover, every inseparable graph with more than two vertices
is a cycle or can be obtained from a cycle by successively adding ears.
\end{theorem}
In particular, as we construct inseparable graphs by adding ears, we do not need to 
add ears to graphs that are separable.

\section{Proof for numbers from 1 to 19}

First, we prove Theorem~\ref{conjecture} is true for numbers from 1 to 3.  For $c(G)=1$, $G$ is a cycle.  For higher $c(G)$ we must add an ear, which gives a theta graph.  The theta graph has 3 cycles.  Thus, 2 is impossible and 3 is possible.

Next we state a valuable though simple lemma whose correctness is obvious.

\begin{lemma}[Ear-Path Lemma]\label{ear-path}
Let $G$ be a graph with vertices $v$ and $w$.  If an ear is added to $G$ with endpoints $v$ and $w$, the number of new cycles created equals the number of $vw$-paths in $G$.
\end{lemma}

Thus, if $G$ is inseparable, then by Menger's Theorem adding an ear must create at least two new cycles.

The notation $G/S$ means the contraction of $G$ by $S \subseteq E$,
while $G\setm P$ means $G$ with the edges and internal vertices of the
ear $P$ (but not its endpoints) removed.  A \emph{minor} of $G$ is any contraction of a subgraph.
If $v,w$ are two vertices of an ear $P$, then $P_{vw}$ denotes the subpath
of $P$ whose endpoints are $v$ and~$w$.

\begin{proposition}[Isotonicity]\label{minors}
Let $G$ be inseparable with $P$ an ear of $G$.  Then 
\begin{enumerate}
\item [\rm{(a)}] $c(H) \leq c(G)$ for every minor $H$ of $G$, 
\item [\rm{(b)}] $c(H) < c(G)$ for every proper subgraph $H$ of $G$, and
\item [\rm{(c)}] $c(G/P) = c(G)$ if and only if $G \setm P$ is separable.
\end{enumerate}
\end{proposition}

\begin{proof}
Deleting an edge $e$ from $G$ reduces $c(G)$ by the number of cycles on $e$, which is positive because $G$ is inseparable.

Contracting an ear $P$ with endpoints $v$ and $w$ reduces $c(G)$ by the number of cycles $C$ that contain both $v$ and $w$ but not $P$.  This number is positive if and only if there is no cutpoint separating $v$ and $w$ in $G \setm P$.  The cutpoint or cutpoints make $G \setm P$ separable.  There is no such cutpoint if and only if $G\setm P$ is inseparable.
\end{proof}

We prove a lemma that tells us how to get the most or fewest cycles by adding an ear.  A vertex is \emph{divalent} if its degree is 2, and \emph{multivalent} if it has degree greater than~2.  Note that an inseparable $G$ will have at least two multivalent vertices unless it is a cycle.

\begin{lemma}[Extremality Lemma]\label{extrem}
Considering all subdivisions $\bar{G}$ of an inseparable graph $G$ that is not a cycle:

{\rm(a)}  The maximum number of cycles obtained by adding an ear to
$\bar{G}$ is achieved (perhaps not uniquely) if the ear is added to
a pair of divalent vertices of~$\bar{G}$ in different ears of $\bar{G}$.

{\rm(b)} The minimum number of cycles obtained by adding an ear to
$\bar{G}$ is achieved (perhaps not uniquely) if the ear is added to
some pair of multivalent vertices of~$\bar{G}$.

{\rm(c)} Adding a new ear joining any two vertices in one ear of $\bar{G}$ gives the same number of cycles as adding the new ear joining the endpoints of that ear.
\end{lemma}

\begin{proof}
Suppose the ear is attached to vertices $v$ and $w$ of $\bar{G}$
and that $v$ is a divalent vertex of $\bar{G}$.
If $P$ is the ear of $\bar{G}$ containing $v$ and $x$ is an end of $P$ such that
the part $P_{xv}$ of $P$ from $x$ to $v$ does not include $w$, then 
contracting that part of $P$ effectively moves $v$ to $x$ and,
by Proposition~\ref{minors}(a) does not increase the number of cycles.
This implies part (b).

To prove (c), suppose that $v$ and $w$ lie on the same ear $P$ of $\bar{G}$,
but $v$ is not at the end.  As before, contract $P_{xv}$.
Since $\bar{G}\setm P_{xv}$ is separable, $w$ being a cutpoint,
the number of cycles is unchanged, by Proposition~\ref{minors}(c).

Now we return to part (a). By the proof of part (b), we can assume $v$ and $w$
are divalent vertices.  If they lie on different ears of $\bar{G}$, we are done.
If not, by part (c) we can move $v$ to $x$ without changing the
cycle count.  Then by the proof of part (b) we can move it to a divalent
vertex on an adjacent  ear
without decreasing the cycle count.  This completes the proof.
\end{proof}

A simple example with large $c(G)$ is a graph $\Theta_k$ consisting of $k$ internally disjoint paths between two vertices.  It has exactly $\binom{k}{2}$ cycles.  The case $k=3$ is a theta graph.

Now we prove Theorem~\ref{conjecture} for numbers 4 through 7.
To have more than 3 cycles we must add an ear to a theta graph. 
(Recall from Theorem~\ref{whitney} that we don't need to consider
adding ears to separable graphs.)
The four possible graphs are shown in Figure~\ref{2-ears}.
They have 6 cycles, except for $K_4$ which has 7 cycles.
This also shows that counts 4 and 5 are impossible.

\begin{figure}[htbp]
\begin{center}
\includegraphics[scale=1.25]{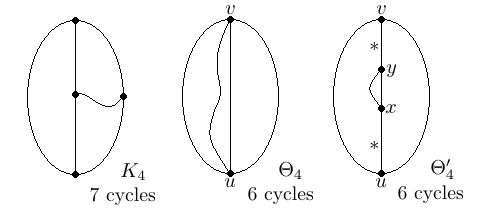}
\caption{The two-ear graphs obtained by adding an ear to a theta graph.  An asterisk $*$ denotes a path that may have length 0.}
\label{2-ears}
\end{center}
\end{figure}

\begin{figure}[htbp]
\begin{center}
\includegraphics[scale=1.25]{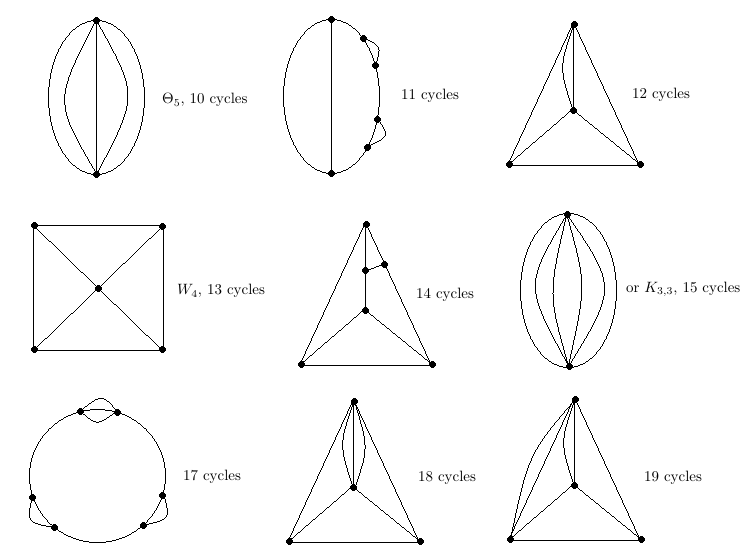}
\caption{Graphs that prove 10, \dots, 15 and 17, 18, 19 are cycle count numbers.}
\label{10-17}
\end{center}
\end{figure}

We now prove Theorem~\ref{conjecture} for cycle counts 8 through 19.
In Figure \ref{10-17} we show graphs that have 10 through 19 cycles, excepting 16.  We now prove that 8 and 9 are not cycle count numbers, as well as establishing which graphs have 10 to 15 cycles.
To get more than 7 cycles we add two ears, called ears three and four, to a 2-ear graph.  The 2-ear graphs are the $\Theta_4'$ graphs and $\bar{K}_4$; they get separate treatment.

We begin with an ear added to $\bar{K}_4$.  Between any two vertices of $K_4$ there are 5 paths so the smallest number of new cycles will be 5, giving 12 cycles.  For the same reason adding a fourth ear will give at least 5 more cycles for a total of 17 or more (exactly 18 if we add a third ear to the double ear in the first diagram in Figure \ref{K4ear}; see Figure \ref{10-17}); in other words, adding two ears to $\bar{K}_4$ cannot give fewer than 17 cycles.  On the other hand, the most new cycles arise from a third ear between internal vertices of two ears of $\bar{K}_4$, which can be done between ears that either do or do not share an endpoint.  The first type yields 14 cycles and the second creates a $\bar{K}_{3,3}$ with 15 cycles.  This establishes that a third ear on $\bar{K}_4$ can yield only between 12 and 15 cycles but a fourth ear yields more than 5 for a total of more than 16 cycles.  The 1-ear extensions of $\bar{K}_4$, with cycle counts from 12 to 15, are shown in Figure~\ref{K4ear}.

\begin{figure}[htbp]
\begin{center}
\includegraphics[scale=1.25]{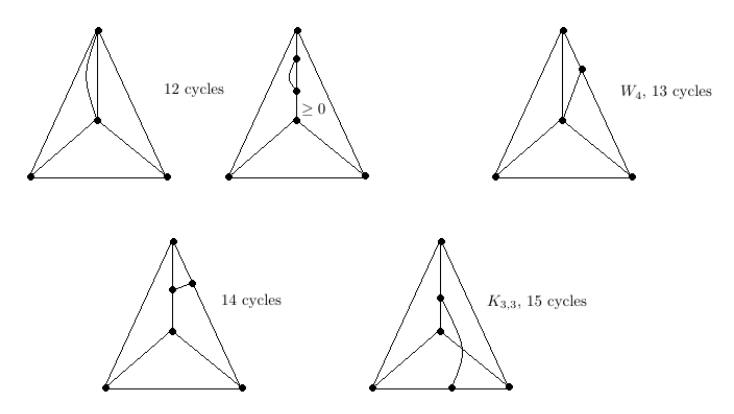}
\caption{The 1-ear extensions of $\bar{K}_4$ and their cycle counts.}
\label{K4ear}
\end{center}
\end{figure}

Adding an ear on the left-most ear of the upper left graph in Figure \ref{K4ear} adds 7 cycles, by Lemma \ref{ear-path}, giving a count of 19 (see Figure \ref{10-17}).

Adding one or two extra ears to a $\Theta_4'$ in a way that creates a
$\bar{K}_4$ subgraph only makes graphs we have already constructed from~$\bar{K}_4$,
so we only need to consider ears that don't create a $\bar{K}_4$ subgraph.

First we add ear three.  The $\Theta_4'$ is built on a $\Theta$, whose constituent paths we call $P, Q, R$, by placing ear two on one of the paths, say $P$.  Let $P = P_{ux} P_{xy} P_{yv}$, where $P_{ux}$ and $P_{yv}$ may have length 0.  (See Figure \ref{2-ears}.)  Each further ear must also be on one of $P$ as in Figure \ref{3-ears}(a2, b) or $Q$ as in Figure \ref{3-ears}(a1), or on ear two, which is equivalent to being on $P$ \emph{nested} inside ear two as in Figure \ref{3-ears}(a2).
In Figure \ref{3-ears}(a) we get $\binom{4}{2} = 10$ cycles because its cycle count is the same as that of $\Theta_5$ by Lemma \ref{extrem}(c).  In Figure \ref{3-ears}(b) we have a necklace of three cycles (it is the upper middle graph in Figure \ref{10-17}); this graph has $11$ cycles.  
We note that adding an ear on a previously added ear does not give any different graphs from these.

\begin{figure}[htbp]
\begin{center}
\includegraphics[scale=1.25]{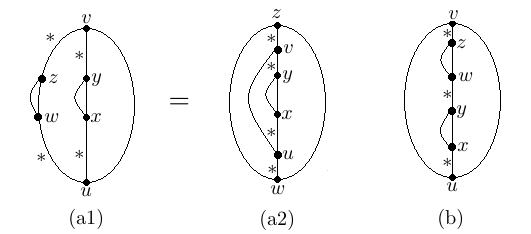}
\caption{The graphs show all ways to add a third ear to $\Theta_4'$ in Figure \ref{2-ears} without creating a $\bar{K}_4$ subgraph.  The labels in (a2) show how it is the same as (a1).  An asterisk $*$ denotes a path that may have length 0.}
\label{3-ears}
\end{center}
\end{figure}

Now we add ear four.  The four essentially different ways to do so are shown in Figure \ref{fourears}.
Graphs (b) show all ways to have ears two, three, and four on the same constituent path of the theta graph; graphs (a) show the other ways.
The four types are represented by (a1--a3) and (b1).
By Lemma \ref{extrem}(c) graphs (a1) and (a3) have the same cycle count as $\Theta_6$, that is, $\binom{6}{2} = 15$.
Graph (b1) is a necklace of four 2-cycles, with a total of 20 cycles.
Graph (a2) has 17 cycles (in fact by Lemma \ref{extrem}(c) it has the same cycle count as the bottom left graph in Figure~\ref{10-17}). 
It is easy to verify that no way of adding ears on ears gives any other graphs.

\begin{figure}[htbp]
\begin{center}
\includegraphics[scale=1.25]{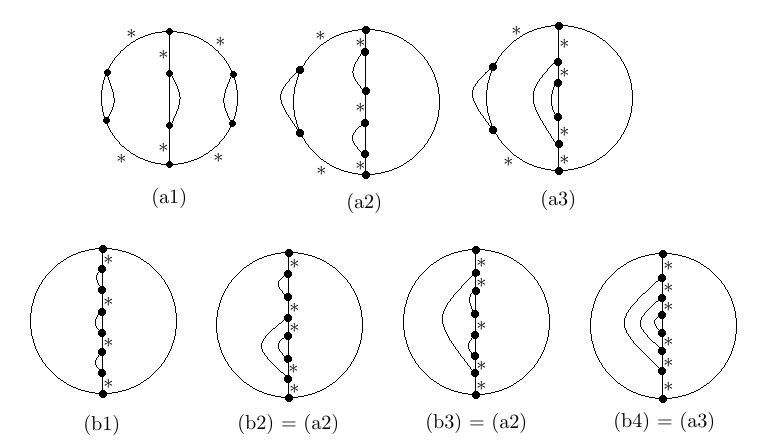}
\caption{Ear additions with four ears.  The graphs show all ways to add a fourth ear to a graph in Figure \ref{3-ears} without creating a $\bar{K}_4$ subgraph.  An asterisk $*$ denotes a path that may have length 0.}
\label{fourears}
\end{center}
\end{figure}

The only way to get a cycle count of 16 is to add a fifth ear to a 4-ear graph with fewer than 16 cycles, but the only count less than 16 is 15, and according to Lemma \ref{extrem}(b, c) a fifth ear adds far more cycles than the 1 that could give a total of 16.  

This proves that 16 cycles are impossible of attainment.  We have also seen that 17, 18, and 19 are possible.

\section{Proof for numbers beyond 89}\label{s:biguns}

We prove this by a general connection between trees and cycle counts.  We construct a bijection between the subtrees of a tree and a graph that has one cycle for every subtree, and we apply a theorem of Czabarka, Sz\'ekely, and Wagner \cite{csw}.

A graph is \emph{outerplanar} if it has a plane embedding with all vertices on the outer face.  The \emph{inner dual} of an inseparable outerplanar graph is the planar
dual without the outside face. It is well known that the
inner dual is a tree. The following converse is also true:

\begin{lemma}\label{inner}
   Let $T$ be a tree with at least two vertices.
   Then there is an inseparable outerplanar graph whose
   inner dual is isomorphic to $T$.
\end{lemma}

\begin{proof}
   Draw $T$ in the plane surrounded by a circle. For each edge
   $e$ of $T$, add a chord to the circle that crosses $e$ once but
   does not otherwise intersect $T$ or another chord (neither internally nor at an endpoint).
   This is always possible, by induction: let $T$ be a tree; remove a leaf from $T$, apply the
   induction hypothesis to add chords, then return the leaf and
   add one more chord that crosses the leaf edge.
\end{proof}

Figure~\ref{cyclegraph} shows an example.
By choosing chords without common endpoints we have obtained
an inseparable outerplanar (and therefore Hamiltonian) cubic graph.

\begin{figure}[htb]
 \centering
 \includegraphics[scale=1.0]{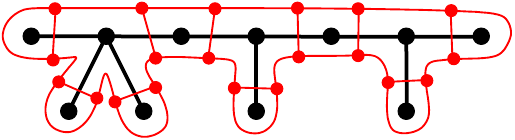}
 \caption{An outerplanar graph (thin red lines) and the tree (heavy black lines) which is its inner dual.}
 \label{cyclegraph}
\end{figure}

\begin{theorem}\label{prop:cor}
   Let $G$ be an inseparable outerplanar graph and let $T$ be its inner dual.
   Then the number of cycles in $G$ equals the number of subtrees of $T$.
\end{theorem}

\begin{proof}
   A cycle of $G$ together with its interior is an outerplanar graph,
   so its inner dual is a subtree of $T$. We show that this gives a
   bijection.  Every cycle in a planar graph is the set sum (the symmetric difference) of the
   boundary edges of the faces inside it. Let $S$ be a subtree of $T$ and consider the set sum
   of the corresponding face boundaries of $G$. The vertices of $S$ can be ordered
   so that each vertex apart from the first is adjacent to exactly one
   previous vertex. Performing the set sum of the corresponding face boundaries
   of $G$ in that order, each face apart from the first shares an edge
   with only one of the previous faces, and that edge is unique.  This
   means we have a cycle at every stage, and moreover the final cycle
   with its interior has $S$ as its inner dual.
\end{proof}

Now we return to the proof of the main theorem.  
It was shown in \cite{csw} that every integer other than 2, 4, 5, 7, 8, 9, 12, 13, 14, 16, 18, 19, 22, 23, 26, 27, 29, 31, 33,
35, 38, 39, 42, 43, 46, 50, 52, 54, 60, 65, 68, 72, 77 and 89 counts the number of subtrees of a tree.  (Cf.\ \cite[Sequence A184164]{oeis}.)
By Theorem~\ref{prop:cor}, every positive integer other than those (and 1) is the number of cycles in an inseparable graph.

\section{Proof for the missing numbers up to 89}

\newcommand{\littlefig}[5]{\parbox{2.5cm}{
 \leavevmode%
 \hbox{\hss\raisebox{#1}{\includegraphics[scale=#2]{#3.pdf}}\hss}
 \\[0.5ex]\centerline{\raisebox{#4}{$c=#5$\quad}}}}

\begin{figure}[ht!]
\centering

\centerline{
  \littlefig{0cm}{.35}{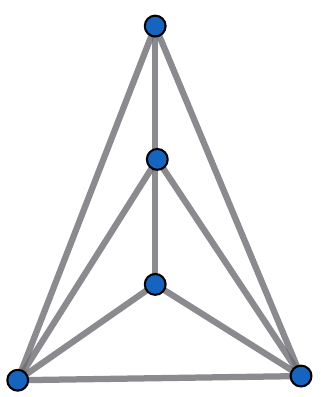}{0cm}{22}
  \littlefig{0cm}{.35}{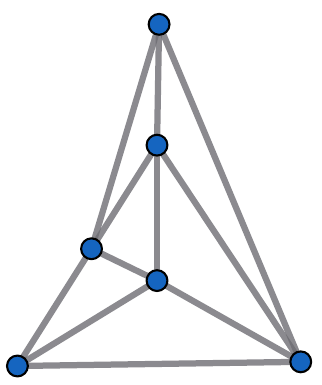}{0cm}{23}
  \littlefig{0cm}{.35}{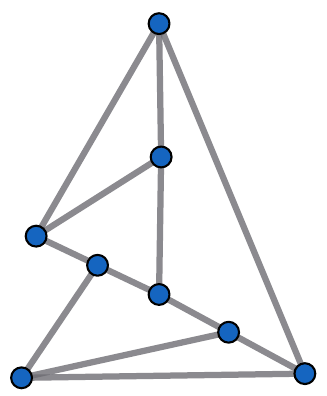}{0cm}{26}
  \littlefig{0cm}{.25}{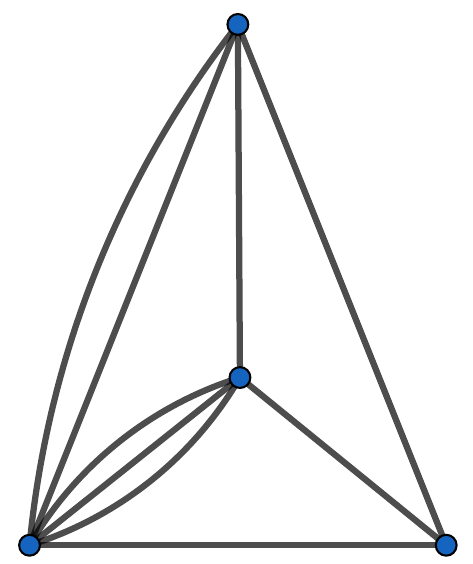}{0cm}{27}
  \littlefig{0cm}{.25}{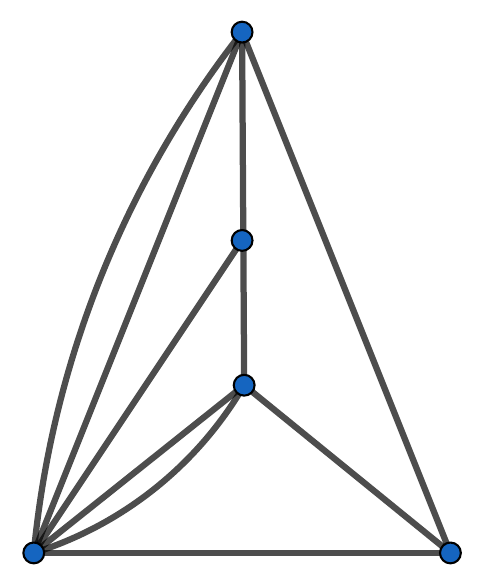}{0.1cm}{29}
  \littlefig{0cm}{.35}{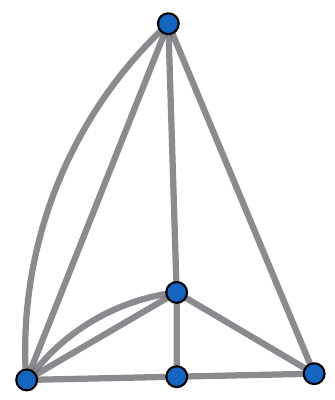}{0cm}{31}
}
\bigskip  
\centerline{
  \littlefig{0cm}{.35}{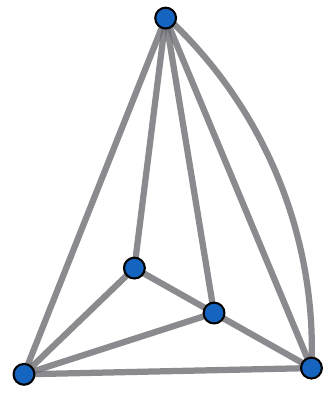}{0cm}{33}
  \littlefig{0cm}{.35}{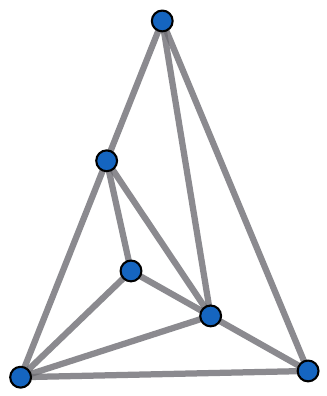}{0cm}{35}
  \littlefig{0cm}{.35}{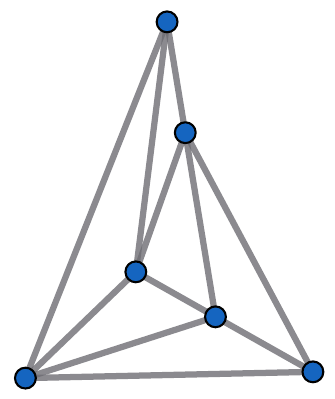}{0cm}{38}
  \littlefig{0cm}{.35}{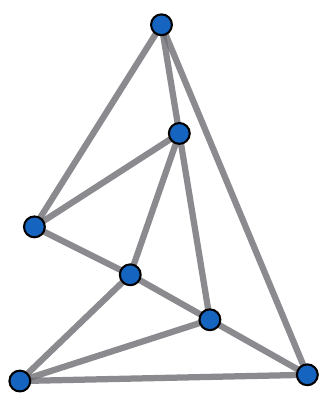}{0cm}{39}
  \littlefig{0cm}{.35}{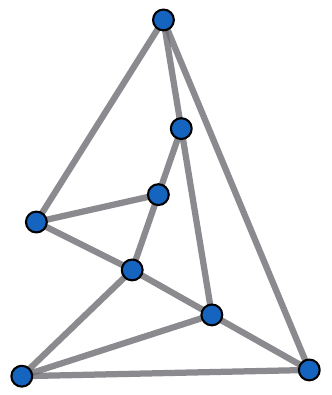}{0cm}{42}
  \littlefig{0cm}{.35}{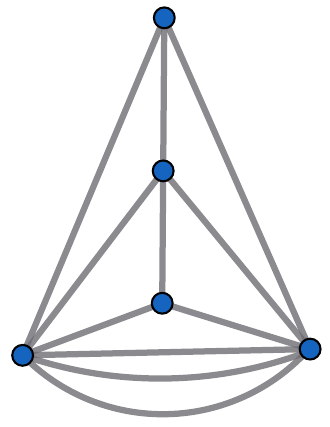}{0.1cm}{43}
}
\bigskip
\centerline{
  \littlefig{0cm}{.35}{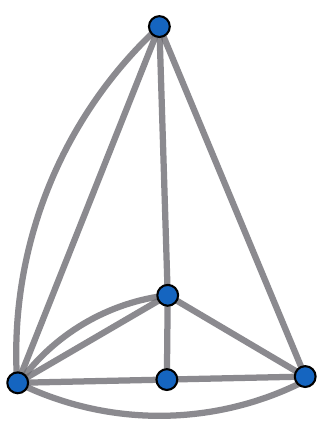}{0cm}{46}
  \littlefig{0cm}{.35}{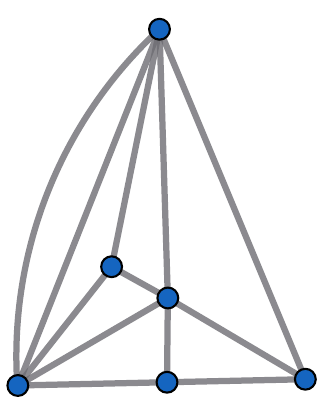}{0cm}{50}
  \littlefig{0cm}{.35}{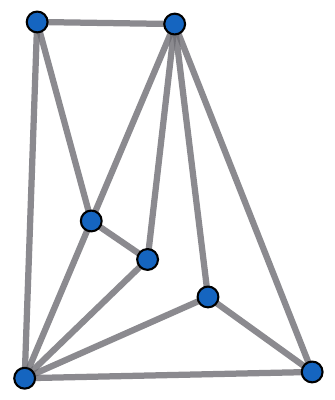}{-0.1cm}{52}
  \littlefig{0cm}{.35}{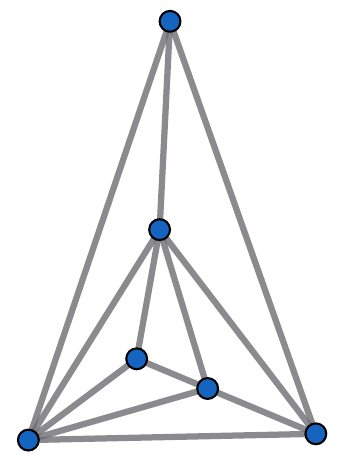}{0.1cm}{54}
  \littlefig{0cm}{.35}{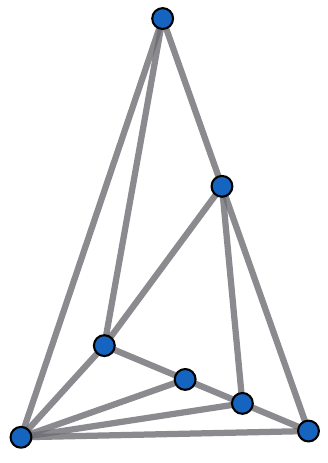}{0.1cm}{60}
  \littlefig{0cm}{.35}{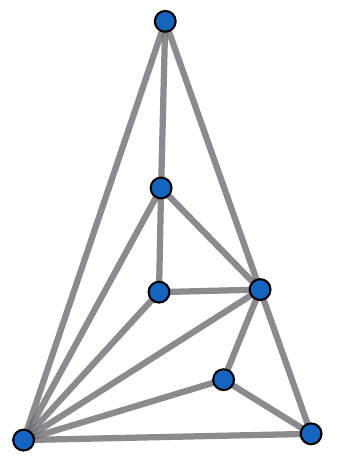}{0.1cm}{65}
}
\bigskip
\centerline{
  \littlefig{0cm}{.35}{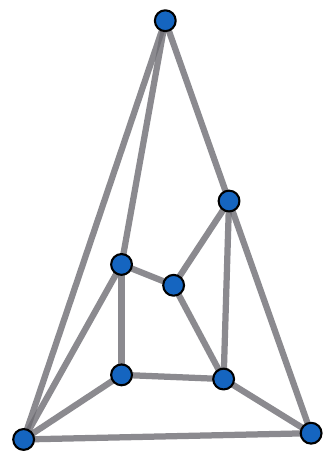}{0cm}{68}
  \littlefig{0cm}{.35}{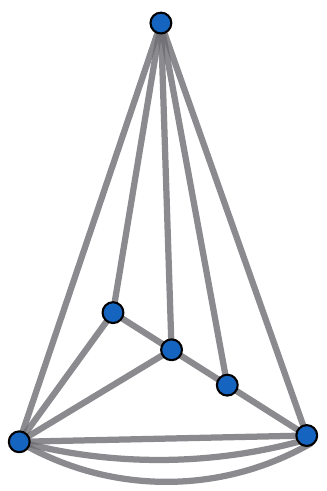}{0.1cm}{72}
  \littlefig{0cm}{.35}{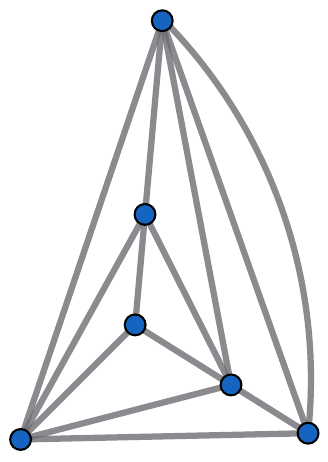}{0cm}{77}
  \littlefig{0cm}{.35}{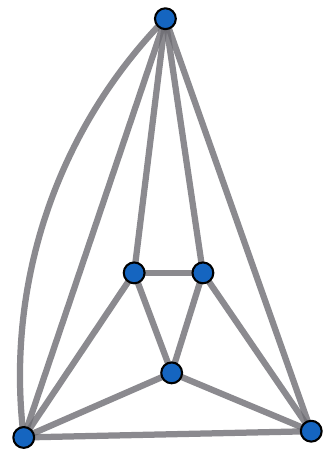}{0cm}{89}
}
\medskip
\caption{Graphs for the 22 exceptional cycle counts from 22 to 89.\label{missing}}
\end{figure}

A computer search using \texttt{nauty}~\cite{nauty} showed that all numbers from 17 up to 100 are cycle counts.  (For computer-friendly data see~\cite{data} under the heading ``Graphs with given cycle counts''.)
We present graphs with the missing cycle counts 22, 23, 26, 27, 29, 31, 33, 35, 38, 39, 42, 43, 46, 50, 52, 54, 60, 65, 68, 72, 77, and 89 in Figure \ref{missing}.  That concludes the proof of Theorem~\ref{conjecture}.

\section{Proof of Theorem~\ref{cubic}}

We first need to show that 13 cycles are not possible.
To have more than 3 cycles, an ear has to be added to the theta graph.  Those graphs are of type $\Theta_4'$ or $\bar{K}_4$ (see Figure~\ref{2-ears}).  The single-ear extensions of $\bar{K}_4$ are shown in Figure~\ref{K4ear}; the only one containing 13 cycles is a $\bar{W}_4$, which is not cubic.  Adding a fourth ear will give 18 or more cycles.

The single-ear extensions of $\Theta_4'$ have either 10 or 11 cycles, and adding another ear will yield either 15, 17, or 20 cycles (see Figure 4), and hence 13 cycles are never achieved in this case.

The construction in Section~\ref{s:biguns} produces inseparable, planar cubic graphs for cycle counts not listed in~\cite{csw} (\cite[Sequence A184164]{oeis}), as we noted there.
This leaves us with the counts listed in~\cite{csw} apart from 1, 2, 4, 5, 8, 9, 13, and 16.
A computer search using \texttt{plantri}~\cite{plantri} found planar cubic graphs for these cases, with up to 14 vertices.
These can be found at~\cite{data}. 
An example which requires 14 vertices is shown in Figure~\ref{cubic68}.

\begin{figure}[htb]
\centering
 \includegraphics[scale=0.7]{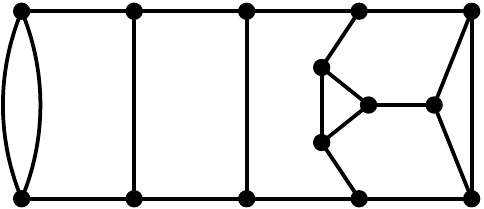}
 \caption{One of the smallest planar cubic graphs with 68 cycles.}
 \label{cubic68}
\end{figure}
 
\section{Consequences}

Curiously, all the graphs we use to generate the cycle counts are planar (except $K_{3,3}$, but it has the planar alternate $\Theta_5$).  They are also Hamiltonian (those in the top row of Figure~\ref{fourears} are not needed for their cycle counts).  That means Theorem \ref{conjecture} is valid for the restricted class of Hamiltonian planar graphs.  

\begin{corollary}
Every positive integer is the cycle count of some inseparable Hamiltonian planar graph, except for $2$, $4$, $5$, $8$, $9$, and $16$.
\end{corollary}

A different kind of consequence is the following.

\begin{corollary}
The only inseparable graphs with fewer than $24$ cycles that are not planar are $K_{3,3}$, with $15$ cycles, and its subdivisions.
\end{corollary}

\begin{proof}
Apply Kuratowski's Theorem along with the following facts: $c(K_5) = 37$, $c(K_{3,3}) = 15$, and with one ear $e$ added to a subdivision we have $c(\bar{K}_{3,3}+e) \geq 24$.  The latter follows from the fact in Lemma \ref{extrem}(b) that adding an edge to $\bar{K}_{3,3}$ gives the fewest new cycles (9, to be exact) if the edge joins two multivalent vertices.
\end{proof}

\section{Conjectures}\label{s:conjectures}

We are not aware of solutions for other classes of graphs.
To motivate such research, we present the results of some computations.
In all cases, it is our conjecture, in the absence of proof, that there are no more missing cycle counts.
\begin{itemize}
  \item[(a)] 3-connected simple graphs. Cycle counts 1--6, 8--12, 16--20, 27, 30, 
    32--34 are missing.  All other counts up to 3000 do occur.
  \item[(b)] 2-connected simple cubic graphs. Cycle counts 1--6, 8--13, 23--25, 27,
   31--37, 40, 43--45, 48, 58--59, 62, 64, 68, 75 and 115 are missing.
   All other counts up to 4000 do occur.
  \item[(c)] 3-connected simple cubic graphs. Cycle counts 1--6, 8--13, 16--25, 27,
  30--45, 48--49, 58--78, 80, 82, 88, 92, 103, 108, 110--132, 134--136, 138--139,
  142--143, 150, 162, 195, 203, 208, 210--212, 214--220, 222--227, 230--233,
  235, 238--239, 243, 247, 251, 255, 403, 407, 411, 415, 419, 423, 427 and 459
  are missing.  All other counts up to 3000 do occur.
  \item[(d)] 2-connected planar cubic graphs.  In addition to the missing counts
  in item (b), this classes misses cycle counts 15--21, 29, 51, 54, 56, 67, 70, 78,
  86, 90, 107, 112, 123 and 131. All other cycle counts up to 4000 do occur.
  \item[(e)] 3-connected planar cubic graphs.  In addition to the missing counts in
  item (d), this class misses cycle counts 15, 29, 51, 53--57, 86, 90, 93, 85, 97--102,
  and many larger counts including 1182. All cycle counts from 1183 to 4000 do occur.
\end{itemize}

\section{Declarations}

McKay was supported by Australian Research Council grant DP250101611.  No other funding supported the preparation of this manuscript.

The authors have no competing interests.

Data for inseparable planar graphs of order at most 8 having all cycle counts up to 100 are available at \cite{data} under the heading ``Graphs with given cycle counts''.
Only distributed software was used for this project.
Graph generators were \texttt{geng} for general graphs~\cite{nauty},
\texttt{plantri} for planar graphs~\cite{plantri} and \texttt{snarkhunter} for cubic graphs~\cite{snarkhunter}.
The tool \texttt{countg} was used for cycle counting~\cite{nauty}.



\end{document}